\documentclass{article}

\usepackage[ngerman, english]{babel}
\usepackage[utf8]{inputenc}
\usepackage{amsmath}
\usepackage{amssymb}
\usepackage{amsthm}
\usepackage{cite}
\usepackage{nicefrac}
\usepackage{hyperref}
\usepackage{booktabs}
\usepackage[]{todonotes} 
\usepackage{enumitem}
\usepackage{mathtools}
\usepackage{faktor}
\usepackage{doi}
\usepackage{geometry}

\newcommand{\kD}{\ensuremath{k(\bar{D})}}

\newcommand{\lB}{\ensuremath{l(\bar{B})}}
\newcommand{\Gl}{\ensuremath{\operatorname{GL}_n(q)}}
\newcommand{\Gu}{\ensuremath{\operatorname{GU}_n(q)}}
\newcommand{\Gleps}{\ensuremath{\operatorname{GL}_n(\varepsilon q)}}
\newcommand{\Sleps}{\ensuremath{\operatorname{SL}_n(\varepsilon q)}}
\newcommand{\Glw}{\ensuremath{\operatorname{GL}_{wd}(q)}}
\newcommand{\Glwo}{\ensuremath{\operatorname{GL}_{w}(\varepsilon q)}}

\newcommand{\Glweps}{\ensuremath{\operatorname{GL}_{wd}(\varepsilon q)}}

\newcommand{\Sl}{\ensuremath{\operatorname{SL}_n(q)}}
\newcommand{\Su}{\ensuremath{\operatorname{SU}_n(q)}}

\newcommand{\Slweps}{\ensuremath{\operatorname{SL}_{w}(\varepsilon q)}}

\renewcommand{\mod}{\operatorname{mod}}

\theoremstyle{definition}
\newtheorem{definition}{Definition}
\newtheorem{Remark}[definition]{Remark}

\theoremstyle{plain}
\newtheorem{theorem}[definition]{Theorem} 
\newtheorem{lemma}[definition]{Lemma}
\newtheorem*{theorem*}{Theorem}

\numberwithin{definition}{section}

\title{On the Malle-Navarro Conjecture for 2- and 3-blocks of general linear and unitary groups}
\author{Sofia Brenner}
\date{}

\begin{document}

\maketitle	

\begin{abstract}
The Malle-Navarro conjecture relates central block theoretic
invariants in two inequalities. In this paper, we prove the conjecture for the $2$-blocks and the unipotent $3$-blocks of the general linear and unitary groups in non-defining characteristic.
Moreover, we show that the conjecture holds for the unipotent $3$-blocks of quotients of
central subgroups of the special linear and unitary groups.
\end{abstract}

Let $G$ be a finite group and consider a block $B$ of the group algebra $FG$ over a field $F$ of characteristic $\ell > 0$. Let $k(B)$ denote the number of irreducible (ordinary) characters in $B$ and $l(B)$ its number of irreducible Brauer characters. Let $D$ be a defect group of $B$ and $k(D)$ the number of conjugacy classes of $D$ and similarly for its derived subgroup $D'$. Finally, $k_0(B)$ denotes the number of characters of height zero in $B$. In \cite{MAL06}, it is conjectured that the following two inequalities hold:

\begin{equation}
k(B) \leq k_0(B) \cdot k(D') \tag{C1}
\end{equation}

and 
\begin{equation}
k(B) \leq l(B) \cdot k(D) \tag{C2}.
\end{equation}

The conjecture was proven in \cite{MAL18} for the $\ell$-blocks
of finite quasi-simple groups where $\ell \geq 5$ is a prime as well as for covering groups of
alternating, sporadic and simple groups of Lie type in defining characteristic. Here, we examine the conjecture for general linear and unitary groups for $\ell \in \{2,3\}$ in non-defining characteristic. Moreover, we treat the special linear and unitary groups in case $\ell = 3$. Our main result is: 

\begin{theorem*}
Let $q = p^f$ for some prime $p$ and $f \geq 1$. 
\begin{enumerate}
\item Assume that $p$ is odd. Then (C1) and (C2) hold for the $2$-blocks of the general linear groups $\Gl$ and the general unitary groups $\Gu$. 
\item Let $p \neq 3$. Then (C1) and (C2) hold for the unipotent $3$-blocks of $\Gl$ or $\Gu$. Moreover, both inequalitites hold for groups of the form $G/Z$ where $G$ denotes either the special linear group $\Sl$ or the special unitary group $\Su$ and $Z \leq Z(G).$ 
\end{enumerate} 
\end{theorem*}

Since the proof for the classical groups of \cite{MAL18} also includes the case $\ell =3$, it only remains to consider the $3$-blocks of the exceptional groups as well as the non-unipotent blocks to prove the conjecture for all finite quasi-simple groups of Lie type in this case. 
\bigskip

This paper is organized as follows: In the first chapter, we present some reduction theorems as well as combinatorial formulas for the block theoretic invariants given above. In the subsequent chapter, we prove the conjecture for the general linear and unitary groups first for the case $\ell= 3$ and a particular case for $\ell =2$, since they behave similarly, before treating the second case for $\ell = 2$ separately. In the last chapter, we consider the special linear and unitary groups for $\ell = 3$. 
\bigskip

\section{Preliminaries and Reductions}
In this chapter, we introduce some first reductions as well as the notation, which will be similar to that in \cite{MAL18}. Throughout, let $\ell \in \{2,3\}$ and $q = p^f$ for a prime $p \neq \ell$ and some $f \geq 1$. As customary, we write $\operatorname{GL}_n(-q)$ for the general unitary group $\Gu$ (and similarly for the special linear group). For $\varepsilon \in \{\pm 1\}$, let $d$ be the order of $\varepsilon q$ modulo $\ell$ (so for $\ell = 2$, we simply have $d = 1$) and $\ell^a$ be the exact power of $\ell$ dividing $(\varepsilon q)^d -1$. 

\begin{Remark}
Let $B$ be a $2$-block of $G = \Gleps$, where $\varepsilon \in \{\pm 1\}$.  The block theoretic invariants of $B$ are the same as for the principal block $B'$ of $C_G(s)$, where $s$ is the semisimple $2'$-element of $G$ corresponding to $B$ by \cite[Prop. 3.4]{BRO86} (see \cite[Cor. 6.4]{GEC91} for the Brauer characters). By \cite[Chapter 1]{FON82}, this is in turn a product of general linear and unitary groups. Since all occurring invariants are multiplicative, it suffices to consider the principal block of $\Gleps$ for the purpose of proving our main theorem. For $ \ell = 3$ and $B$ a unipotent block of $FG$, there exists a weight $w > 0$ such that the block theoretic invariants of $B$ are the same as those of the principal block $B'$ of $\operatorname{GL}_{wd}(\varepsilon'q)$ for some $\varepsilon' \in \{\pm 1\}$ (cf. \cite[Thm 1.9]{MIC83}).
\end{Remark}

In the following, we therefore may assume that $B$ is the principal $\ell$-block of $G = \Glweps$, where $q$ is not divisible by $\ell$, $\varepsilon \in \{\pm 1\}$ and $w \geq 1$. 
\bigskip

Let $s \in \mathbb{Z}_{> 0}$ and $t \in \mathbb{Z}_{\geq 0}$. By $\pi(t)$, we denote the number of \emph{partitions} of $t$ and write $|\lambda| = t$ if $\lambda$ is a partition of $t$. By $k(s,t)$, we denote the number of $s$-\emph{multipartitions} of $t$, that is, the number of tuples $(\mu, \ldots, \mu_s)$ of partitions $\mu_1, \ldots, \mu_s$ such that $|\mu_1| + \ldots + |\mu_s| = t$. Furthermore, we define an $\ell$-\emph{decomposition} of $t$ to be a tuple $(t_0, \ldots, t_k)$ of nonnegative integers $t_0, \ldots, t_k$ such that $\sum_{i = 0}^k t_i \ell^i = t$ and $t_k \neq 0$. The set of $\ell$-decompositions of $t$ will be denoted by $W_t$ and its cardinality by $p_\ell(t)$. Furthermore, an ordered tuple $(t_1, \ldots, t_s)$ of nonnegative integers with $t_1 + \ldots + t_s = t$ is called an $s$-\emph{split} of $t$ (write $\lambda \Vdash t$). 
\bigskip

For any natural number $n$ and a prime number $r$, denote by $n_r$ the largest power of $r$ dividing $n$. Let $w = \sum_{i = 0}^v a_i \ell^i$ be the $\ell$-adic decomposition of $w$. We recall the values of some invariants: 

\begin{lemma}\label{lem:kBformula}
Let $2^{\tilde{a}} = (q+ \varepsilon)_2$. For the principal $2$-block $B$ of $\Glwo$, it holds that
\begin{equation}\label{eq:formulakB2}
k(B) = \sum_{ \textbf{w} \in W_w} k(2^a, w_0) \, k(2^{a+\tilde{a}-1}-2^{a-1},w_1) \cdot \prod_{i \geq 2} k(2^{a+\tilde{a}-2}, w_i).
\end{equation}
Let $b = d + \frac{3^a - 1}{d}$ and $b_1 = 2 \cdot \frac{3^{a-1}}{d}$. For the principal $3$-block $B$ of $\Glw$, it holds that
\begin{equation}\label{eq:formulakB3}
k(B) = \sum_{\textbf{w} \in W_w} k(b, w_0) \cdot \prod_{i \geq 1} k(b_1, w_i) =: k(3,a,d,w).
\end{equation} 
\end{lemma}
\begin{proof}
Cf. \cite[Prop. 2.39 and Lemma 2.44]{GRU18} for $\ell =2$ and \cite[Prop. 6]{OLS84} for $\ell = 3$. 
\end{proof}

Observe that the formulas for $\ell= 2$, $\varepsilon q \equiv 1 \mod 4$ and $\ell = 3$ are similar, so we treat these cases in parallel. Moreover, we write $k^w(B)$ if we want to clarify which value of $w$ is currently examined. 

\begin{lemma}
The number of characters of height zero in the principal $\ell$-block $B$ is given by 
$k_0(B) = 2^{\sum_{i = 0}^v a_i(i+1)}$ for $\ell = 2$ and by $k_0(B) = \prod_{i \geq 0} k(b \cdot 3^i, a_i)$ for $\ell = 3$. 
\end{lemma}
\begin{proof}
Cf. \cite[Thm.~2.60]{GRU18} for $\ell = 2$ and \cite[Prop. 2.13]{MIC83} for $\ell = 3$. 
\end{proof}

In the following, denote by $\operatorname{SD}_{2^{\tilde{a}+2}} = \langle x,y \mid x^2 = y^{2^{\tilde{a}+1}} = 1, \, xyx = y^{2^{\tilde{a}}-1} \rangle$ the semidihedral group of order $2^{\tilde{a}+2}.$ The defect groups $D$ of the principal $\ell$-block are Sylow $\ell$-subgroups of $G$ whose structure can be described as follows: 

\begin{lemma}
	\begin{enumerate}
		\item Let $\ell = 2$ and $\varepsilon q \equiv 1 \mod 4$, or $\ell = 3$. Then $D \cong \prod_{i = 0}^v D_{i,\ell^a}^{a_i}$, where $D_{i,\ell^a} = C_{\ell^a} \wr C_\ell \wr \ldots \wr C_\ell$ is the iterated wreath product of the cyclic group $C_{\ell^a}$ with $i$ factors of the cyclic group $C_\ell$. 
		\item If $\ell = 2$ and $\varepsilon q \equiv 3 \mod 4$, then $D \cong \prod_{i = 0}^{a_i} P_{2^i}^{a_i}$, where $P_1 = C_2$ and for $i \geq 1$, we have $P_{2^i} = \operatorname{SD}_{2^{\tilde{a}+2}}\, \wr  \, C_2 \wr \ldots \wr C_2$ with $i-1$ factors of the cyclic group $C_2$. 
	\end{enumerate}
\end{lemma}
\begin{proof}
See \cite[p.18]{GRU18} for $\ell = 2$ and \cite[Prop. 5.11]{MAL18} for $\ell = 3$. 
\end{proof}

\section{General Linear and Unitary Groups}
In this chapter, we prove the inequalities (C1) and (C2) for the general linear and unitary groups $\Gleps$, using the notation from the previous chapter. We first assume $\ell = 3$ or $\varepsilon q \equiv 1 \mod 4$ if $\ell = 2$, that is, $a \geq 2$ and $\tilde{a} = 1$. We begin by deriving bounds for the occurring numbers of multipartitions. 
\begin{lemma}\label{lem:multipartitionsbasics}
\begin{enumerate}
\item For all $s \geq 3$ and $t \geq 1$, it holds that $k(s,t) \leq s^t$. 
\item For all $s\geq 3$ and $t_1, t_2 \geq 1$, it holds that $k(s, t_1 + t_2) \leq k(s,t_1) \cdot k(s,t_2).$ Moreover, for $t \geq 2$, it holds that $k(2,t+1) \leq 2 \cdot k(s,t).$
\end{enumerate}
\end{lemma}
\begin{proof}
Cf. \cite[Lemma 5.5]{MAL18} and \cite[Lemma 2.48]{GRU18} for the second part of (ii).
\end{proof}

\begin{lemma}\label{lem:k3aw2}
For $s \geq 1$, it holds that $k(s,1) = s$, $k(s,2) = \frac{1}{2}s^2 + \frac{3}{2} s$ and $k(s,3) = \frac{1}{6}s^3 + \frac{3}{2}s^2 + \frac{4}{3} s.$
\end{lemma}
\begin{proof}
By \cite[Lemma 1]{OLS84} it holds for all $t \geq 0$ that $k(s,t) = \sum_{(k_1, \ldots, k_s) \Vdash t} \pi(k_1) \cdots \pi(k_s),$ so counting the different $s$-splits of $t \in \{1,2,3\}$ and using $\pi(t) = t$ in this case yields the claim. 
\end{proof}

\begin{lemma}\label{lem:multipartitions}\label{lem:k2tschwach}
Let $w \geq 0$.
\begin{enumerate}
\item It holds that $k(2,w) \leq 2^{w + 0.35}.$
\item For $a \geq 3$, it holds that
$k(2^a, w) \leq 2^{\left(a-\frac{4}{3}\right)w + 3}.$ For $a\geq 5$, we have $k(2^a, w) \leq 2^{\left(a-\frac{4}{3}\right)w + 2}.$
\item For $a \geq 2$, it holds that 
$k(b, w) \leq 3^{\left(a-\frac{5}{6}\right)w + 2- \log_3(d)}$, where $b$ is defined as in Lemma \ref{lem:kBformula}. For $a \geq 3$ and $w \geq 9$, one can omit the summand 2 in the exponent. 

\end{enumerate}
\end{lemma}

\begin{proof}
Using Lemma \ref{lem:multipartitionsbasics}, the first claim follows $k(2,2) = 5 \leq 2^{2.35}$ by induction. Now consider the second inequality. For $c \geq 3$ and $x \geq 1$, it holds that 
\begin{equation}\label{eq:k(cx,w)}
k(cx, w) = \sum_{(i_1, \ldots, i_x) \vDash w} k(c, i_1) \cdots k(c,i_x)  \leq \sum_{(i_1, \ldots, i_x) \vDash w} c^w = \binom{x + w-1}{w} c^w
\end{equation} 
(cf. \cite[Lemma 5.6]{MAL18}). We apply this estimate with $x = 4$. To this end, we claim that for all $w \geq 0$, it holds that 
$$\binom{w+3}{w} \leq 2^{\frac{2}{3}w + 3}.$$
For $w \leq 5$, this can be checked directly. For $w \geq 5$, we obtain by induction
$$\binom{(w+1)+3}{w+1} = \frac{w+4}{w+1} \cdot  \binom{w+3}{w} \leq 2^{\frac{2}{3}(w+1) + 3}$$
since $\frac{w+4}{w+1} = 1 + \frac{3}{w + 1} \leq 2^{2/3}$ for $w \geq 5$. Equation \eqref{eq:k(cx,w)} then yields for $a \geq 4$
$$k(2^a, w) \leq \binom{w+3}{w} \cdot 2^{(a-2)w} \leq 2^{(a-2)w + \frac{2}{3}w + 3} = 2^{(a-\frac{4}{3}) w + 3}.$$
For $a = 3$, we check the claim directly for $w \leq 9$ using GAP \cite{GAP4}. For $w \geq 10$, we can use the above proof to show that even $\binom{w+3}{w} \leq 2^{2/3 w + 1.6}$, so with the first part of the lemma we obtain
$$k(8,w) = \sum_{(i_1,\ldots, i_4) \vDash w} k(2, i_1) \cdots k(2, i_4) \leq \sum_{(i_1,\ldots, i_4) \vDash w} 2^{i_1 + 0.35 + \ldots + i_4 + 0.35} = \binom{w+3}{w} \cdot 2^{w+ 1.4} \leq 2^{\frac{5}{3}w + 3}.$$

For the stronger bound for $k(2^a,w)$, we use $x = 8$ instead of $x = 4$. The last part of the lemma can be proven in the same fashion by using $x = 3$ and $b = 3^a$ for $d = 1$ and $b = (3^a + 3)/2 \leq 2 \cdot 3^{a-1}$ for $ d = 2$. 
\end{proof}

\begin{lemma}\label{lem:bound3,4}
For $w \geq 1$, it holds that $k(3,w) \leq 3^{\frac{w}{2} + \frac{9}{4}}$ and $k(4,w) \leq 2^{1.2 w + 2} $. 
\end{lemma}

\begin{proof}
Using $\pi(n) \leq \frac{e^{c \sqrt{n}}}{n^{3/4}}$ with $c = \pi \sqrt{2/3}$ (cf. \cite[p.114]{AZE09}), we obtain
$\pi(n) \leq 1.4^{n + 1.2}$ for $n \geq 38$. We can check directly that this bound in fact holds for all $n \geq 1$. With this, we have 
	$$k(3,w) = \sum_{(i_1, i_2, i_3) \vDash w} \pi(i_1) \pi(i_2) \pi(i_3) \leq \binom{w+2}{w} \cdot 1.4^{w + 3.6}.$$
	The last term can be bounded by $3^{\frac{w}{2} + \frac{9}{4}}$ for $w \geq 20$. The remaining cases are checked in GAP.
%
%
\\
By Lemma \ref{lem:k2tschwach}, we have
$$k(4,w) = \sum_{i_1 + i_2 = w} k(2,i_1) k(2,i_2) \leq \sum_{i_1 + i_2 = w} 2^{i_1 + 0.35} \cdot 2^{i_2 + 0.35} = (w+1) \cdot 2^{w+0.7}.$$
With $w+1 \leq 2^{0.2w +1.3}$ for $w \geq 13$, we obtain the desired bound. The remaining cases can be checked in GAP.  
\end{proof}

\begin{lemma}\label{lem:plw}
	\begin{enumerate}
		\item It holds that $p_3(w) \leq 3^{w/6}$ for all $w \neq 3$. 
		\item For $w \geq 0$, it holds that $p_2(w) \leq 2^{\frac{w}{3}+1}.$
	\end{enumerate}
\end{lemma}

\begin{proof}
	By \cite[Lemma 5.2]{MAL17}, it holds that $p_\ell(w) \leq \frac{w}{\ell} \cdot p_\ell(\lfloor w/\ell \rfloor)$ for $\ell \geq 2$ and $w \geq 1$. Checking the claim directly for $w \leq 12$, we can use induction and $w/3 \leq 3^{w/9}$ for $w \geq 8$ to show by induction that 
	$p_\ell(w) \leq \frac{w}{\ell} \cdot p_\ell(\lfloor w/\ell \rfloor) \leq 3^{w/9} \cdot 3^{w/18} = 3^{w/6}.$
	The second statement can be proved in the same fashion.
\end{proof}

\begin{lemma}\label{lem:bound1.5} \label{lem:bound23}
\begin{enumerate}
\item For $\ell = 2$ and $a \geq 4$, it holds that $k(B) \leq 2^{(a-1) w + 3/2}.$ For $a \geq 3$, there is the weaker bound $k(B) \leq 2^{(a-1)w +3}.$ 
\item For $a \geq 2$ and any $3$-decomposition $(w_0, \ldots, w_v)$ of $w$, we have 
$$k(b,w_0) \cdot \prod_{i \geq 1} k(b_1, w_i)	\leq 3^{\left(a-\frac{5}{6}\right)w + 2 - \log_3(d)},$$
which yields $k(B) \leq p_3(w) \cdot 3^{\left(a-\frac{5}{6}\right)w + 2 - \log_3(d)}.$
\end{enumerate}
\end{lemma}

\begin{proof}
We first prove the second part. By Lemma  and \ref{lem:multipartitionsbasics} and \ref{lem:multipartitions}, it holds that
\begin{alignat*}{2}
k(b,w_0) \cdot \prod_{i \geq 1} k(b_1, w_i) &\leq 3^{\left(a-\frac{5}{6} \right)w_0 + 2 - \log_3(d)} \cdot \prod_{i \geq 1} 3^{a w_i} 
&&\leq
3^{\left(a-\frac{5}{6} \right)w_0 +2 - \log_3(d)} \cdot 3^{a \frac{w-w_0}{3}} \\
&= 3^{a \frac{w}{3} + \left(\frac{2}{3}a-\frac{5}{6}\right)w_0 + 2 - \log_3(d)} 
&&\leq 3^{\left(a-\frac{5}{6} \right) w + 2 - \log_3(d)}.
\end{alignat*}	
	
Now consider the first bound for $\ell = 2$. Let $w \geq 2$ and $a \geq 5$. We use the stronger bound from Lemma \ref{lem:multipartitions}. There is a single binary decomposition of $w$ with $w_0 = w$ and at most two with $w_0 = w-2$. For all others, it holds that $w_0 \leq w-4$, since $w_0$ and $w$ must have the same parity. Analogously to the above, we have 
	\begin{alignat*}{1}
	k(B) &= \sum_{ \textbf{w} \in W_w} k(2^a, w_0) \prod_{i \geq 1} k(2^{a-1},w_i) \leq \sum_{\textbf{w} \in W_w} 2^{\left(a-\frac{4}{3}\right) w_0 +2+ (a-1)\frac{w-w_0}{2}} \\
	&\leq p_2(w) \cdot 2^{\left(\frac{a}{2}-\frac{5}{6}\right)(w-4) +2+ \frac{a-1}{2}w}  + 2 \cdot 2^{\left(\frac{a}{2}-\frac{5}{6}\right) (w-2) +2+ \frac{a-1}{2}w} + 2^{\left(\frac{a}{2}-\frac{5}{6}\right)w +2+ \frac{a-1}{2}w} \\
	&\leq 2^{(a-1)w + \frac{3}{2}} \cdot \left(2^{\frac{29}{6} - 2a}+ 2^{- \frac{w}{3} + \frac{19}{6} - a} + 2^{-\frac{w}{3} +\frac{1}{2}} \right) \leq 2^{(a-1)w + \frac{3}{2}}.
	\end{alignat*} 
	Here, we inserted the estimate $p_2(w) \leq 2^{\frac{w}{3}+1}$ (cf.\ Lemma \ref{lem:plw}) in the second step. In the third one, we used that the term in brackets is smaller than one for $a \geq 6$ and $w \geq 2$ as well as for $a =5$ and $w \geq 3$. The finitely many remaining cases can be checked directly. For $a \in \{3,4\}$, we use the same approach, albeit with the weaker bound of Lemma \ref{lem:multipartitions}, to prove the claim.  For $w = 1$, we have $k(B) = k(2^a, 1) = 2^a,$ so the inequality holds.
\end{proof}

We now treat the case of small values of $a$. 

\begin{lemma}\label{lem:bound1.65}
	Let $w \geq 1$. For $\ell =2$ and $a = 2$, we have $k^w(B) \leq 2^{1.4 w + 1.65}$ and for $\ell = 3$ and $ a = 1$, it holds that
	$k^w(B) \leq 3^{\frac{w +7}{2}}.$
\end{lemma}

\begin{proof}
	We prove the claim by induction on $w$. Note that for any $\ell$-decomposition $(w_0, w_1, \ldots, w_v)$ of $w$, $\tilde{w} = (w_1, \ldots, w_v)$ is an $\ell$-decomposition of $(w-w_0)/\ell$ and each of them arises in this way. So there is a bijection between $W_w$ and $\bigcup_{j = 0}^{(w-a_0)/\ell} W_{r(j)}$, where $r(j) = (w-(a_0 + \ell j))/\ell$ (note that $w_0$ and $w$ have the same remainder modulo $\ell$). Summing over all possible values of $w_0$, we therefore obtain (setting $k^0(B) := 1$)
	\begin{alignat*}{1}
	k^w(B)&=  \sum_{\textbf{w} \in W_w} k(\ell^a, w_0) \cdot \prod_{i \geq 1} k(\ell^{a-1}, w_i) = \sum_{j = 0}^{\frac{w-a_0}{\ell}} \sum_{\tilde{w} \in W_{r(j)}} k(\ell^a, \ell j+a_0) \cdot \prod_{i \geq 1} k(\ell^{a-1}, w_i) \\
	&\leq k(\ell^a, w)+\sum_{j = 0}^{\frac{w-a_0}{\ell}-1} k(\ell^a, \ell j +a_0)  \sum_{\tilde{w} \in W_{r(j)}} k(\ell^a, w_1) \cdot \prod_{i \geq 2} k(\ell^{a-1}, w_i) \\
	&= \sum_{j = 0}^{(w-a_0)/\ell} k(\ell^a, \ell j +a_0) \cdot k^{r(j)}(B).
	\end{alignat*}
	
    By induction, using the geometric series as well as the bound from Lemma \ref{lem:bound3,4}, we obtain for $\ell = 2$
	\begin{alignat*}{1}
	k(B) &\leq \sum_{j = 0}^{(w-a_0)/2} 2^{1.2 (2j + a_0) +2} \cdot 2^{1.4 \frac{w- (2j + a_0)}{2} + 1.65} \leq 2^{3.65+ 0.7 w + 0.5 a_0} \cdot \left(2^{\frac{w-a_0}{2}+1} -1\right) \\
	&\leq 2^{1.2 w + 4.65} \leq 2^{1.4 w + 1.65},
	\end{alignat*}
	where the last inequality holds for $w \geq 15$. For $\ell = 3$, assume that $k^i(B) \leq 3^i$ for $i \leq \frac{w-a_0}{3}$. With the above and the bound from Lemma \ref{lem:bound3,4}, we have 
	
	$$k(B) \leq \sum_{j = 0}^{(w-a_0)/3} 3^{\frac{3j + a_0}{2} + \frac{9}{4}} \cdot 3^{\frac{w-(3j  +a_0)}{3}} \leq 3^{\frac{9}{4} + \frac{a_0}{6} + \frac{w}{3}} \cdot \frac{3^{\frac{1}{2}\left(\frac{w-a_0}{3}+1\right)} -1}{\sqrt{3}-1} \leq 3^{\frac{w+7}{2}} \leq 3^w,$$

	
	where the last inequality holds for $w \geq 7$. Checking directly that $k^w(B) \leq 3^w \leq 3^{\frac{w+7}{2}}$ for $w \leq 6$, this shows inductively that $k^w(B) \leq \min\{3^\frac{w+7}{2}, 3^w\}$. The remaining cases can be checked directly.  
\end{proof}

We have now assembled the prerequisites to prove the inequalities (C1) and (C2) for the general linear and unitary groups. To this end, note that by the same argument as in the proof of \cite[Prop. 5.11]{MAL18}, we may assume in the following that $w$ is divisible by $\ell$. For the number of characters in $D'$ and $D$, it holds by \cite[Lemma 5.10]{MAL18}
\begin{equation}\label{eq:kD'3}
k(D') = \prod_{i \geq 1} k(D_{i,\ell^a}')^{a_i} \geq \prod_{i \geq 1} \ell^{a(\ell^i-1) - \frac{\ell^i-\ell}{\ell-1} - i +1} = \ell^{\left(a- \frac{1}{\ell - 1}\right) w - \sum a_i \left(a+ i - \frac{2 \ell - 1}{\ell- 1}\right)},
\end{equation} and 
\begin{equation}\label{eq:kD3}
k(D) = \prod_{i \geq 1} k(D_{i,\ell^a})^{a_i} \geq \prod_{i \geq 1} \left(\frac{\ell^{a \ell^i}}{\ell^{(\ell^i-1)/(\ell-1)}}\right)^{a_i} = \ell^{\left(a-\frac{1}{\ell-1}\right) w + \frac{1}{\ell - 1} \sum_{i \geq 1} a_i}.
\end{equation}

For small values of $a$, we need to improve this bound: 
\begin{lemma}\label{lem:nrcharacters}
For $i \geq 1$, it holds that $$k(D_{i,3}) \geq 3^{3^\frac{i-1}{2}} \cdot 3^{\frac{3^i +1}{2}}$$ 
	and 
		$$k(D_{i,3}') \geq 3^{3^\frac{i-1}{2}} \cdot 3^{\frac{3^i +1}{2}-i}.$$	
\end{lemma}

\begin{proof}
	The proof can be carried out analogously to  \cite[Lemma 5.10]{MAL18} by using $k(D_{1,3}) = 17$ as an improved induction start for the first part.
\end{proof}

\begin{theorem}
(C1) and (C2) hold for the principal $3$-block of $\Glweps$.
\end{theorem}

\begin{proof}
	First consider $\Glw$. 
	As in \cite[Prop. 5.11]{MAL18}, the number of characters of height zero in $B$ can be bounded from below by
		$$k_0(B) = \prod_{i \geq 1} k(b\ell^i,a_i) \geq \prod_{i \geq 1} \left(\frac{b}{3^a}\right)^{a_i} \cdot 3^{\sum_{i \geq 1} a_i (a+i-1)} \geq 3^{\sum_{i \geq 1} a_i (a+i-1 -\log_3(d))}.$$
	Moreover, it holds that $l(B) \geq k(d,w) \geq p_\ell(w)$ (cf. \cite[Prop. 5.11]{MAL18}). 
	First assume $a \geq 2$. For $w \geq 6$, Lemma \ref{lem:bound23} yields
	$$k(B) \leq p_3(w) \cdot 3^{\left(a-\frac{5}{6}\right) w
		+ 2 - \log_3(d)} \leq 3^{\left(a-\frac{1}{2}\right)w  + \left(\frac{3}{2} - \log_3(d)\right) \sum_{i \geq 1} a_i} \leq k_0(B) \cdot k(D')$$ and
	$$k(B) \leq p_3(w) \cdot 3^{\left(a-\frac{5}{6}\right) w
		+ 2- \log_3(d)} \leq \pi(w) \cdot 3^{\left(a-\frac{1}{2}\right)w + \frac{1}{2} \sum a_i} \leq l(B) \cdot k(D),$$
	since $-\frac{w}{3} +2 - \log_3(d) \leq \frac{1}{2} \leq \min\{\left(\frac{3}{2} - \log_3(d)\right), \frac{1}{2}\} \cdot \sum a_i$ for $w \geq 6$. For $w = 3$, the above inequalities remain valid when inserting $p_3(3) = 2$. 
	For $a = 1$, we use the improved bounds from Lemma \ref{lem:Sylowb}
		\begin{equation}\label{eq:kD'31}
	k(D') = \prod_{i \geq 1} k(D_i')^{a_i} \geq \prod_{i \geq 1} \left(3^{3^{\frac{i-1}{2}}} 3^{\frac{3^i +1}{2}-i}\right)^{a_i}= 3^{\frac{2}{3} w + \sum_{i \geq 1} a_i \left(\frac{1}{2}- i\right)},
	\end{equation} and 
	\begin{equation}\label{eq:kD31}
	k(D) = \prod_{i \geq 1} k(D_i)^{a_i} \geq \prod_{i \geq 1} \left(3^{3^{\frac{i-1}{2}}} 3^{\frac{3^i +1}{2}}\right)^{a_i} = 3^{\frac{2}{3}w + \frac{1}{2}\sum_{i \geq 1} a_i}.
	\end{equation}
	
	Moreover, note that $b = 3 = 3^a$ for both $d = 1$ and $d = 2$ in this case. With 
	$$k(3^{i+1}, a_i) = \begin{cases}
	3^{i+1} & \text{ if }a_i = 1 \\
	\frac{3^{2i+2}}{2} + \frac{3^{i+2}}{2} \geq 3^{2i+1} & \text{ if }a_i = 2,
	\end{cases}$$
	we obtain 
	\begin{equation}\label{eq:betterboundk0B}
	k_0(B) =  \prod_{i \geq 1} k(3^{i+1}, a_i)\geq 3^{\sum_{i \geq 1} a_i i + \sum_{i: a_i \neq 0} 1}.
	\end{equation}
	
	For $w \notin \{3,6,9\}$, we have 
	$$k(B) \leq 3^{\frac{w+7}{2}} \leq 3^{\sum_{i \geq 1} a_i i + \sum_{a_i \neq 0} 1 } \cdot 3^{\frac{2}{3}w + \sum_{i \geq 1} \left(\frac{1}{2} -i\right) a_i} \leq k_0(B) \cdot k(D').$$ For $w \geq 6$, we furthermore obtain
	$$k(B) \leq 3^{\frac{w+7}{2}} \leq \pi(w) \cdot 3^{\frac{2}{3}w + \frac{1}{2} \sum a_i} \leq l(B) \cdot k(D),$$ since then $\pi(w) \geq \pi(6) = 11$. In the remaining cases, we check the inequalities directly: 
%
	For $w = 3$, we have $k^3(B) = 24<3^3$ and $k(D') \geq 3^{\frac{3}{2}}$ as well as $k_0(B) = 3^2$. Moreover, $l(B) \geq \pi(3) = 3$ and $k(D) \geq 3^{5/2}$, so also (C2) holds. For $w = 6$, it holds that $k^6(B) = 270<3^6$, $k(D') \geq 3^3$ and $k_0(B) = 54> 3^3$. Finally, in case $w = 9$, we obtain $k^9(B) = 2043<3^7$, $k(D') \geq 3^\frac{9}{2}$ and $k_0(B) = 3^3$. 
	This finishes the proof for $\Glw$.
\bigskip

Denoting the order of $-q$ modulo 3 by
$d$, the block theoretic invariants of $\operatorname{GU}_{wd}(q)$ are the same as of $\operatorname{GL}_{wd}(q_0)$, where $q_0$
has order $d$ modulo 3 and $3^a$ is the exact power of 3 dividing $q_0^d- 1$ (cf. \cite[Prop. 5.11]{MAL18}), so
the claim follows from the proven inequality for the linear case.
\end{proof}	

For $\ell = 2$, the formulas hold for the general linear as well as for the general unitary groups. 	
\begin{theorem}
(C1) and (C2) hold for the principal $2$-block of $\operatorname{GL}_w(\varepsilon q)$ if $\varepsilon q \equiv 1 \mod 4$. 
\end{theorem}

\begin{proof}
	For $a \geq 3$, we use the bound from Lemma \ref{lem:bound23} together with Equations \eqref{eq:kD'3} and \eqref{eq:kD3} to obtain
	$$k(B) \leq 2^{(a-1)w + 3} \leq 2^{(a-1)w  + 3 \sum_{ \geq 1} a_i} \leq   k_0(B) \cdot k(D')$$
	and, for $w \geq 4$,
	$$k(B)\leq 2^{(a-1)w + 3}  \leq \pi(w) \cdot 2^{(a-1)w + \sum a_i} \leq l(B) \cdot k(D),$$
	since then $\pi(w) \geq \pi(4) > 4$. In case $w = 2$, the claim follows similarly for $a \geq 4$ by using the stronger bound from Lemma \ref{lem:bound23}. For $a = 3$ and $w = 2$, we have $k(B) = 48 \leq 2^{6} \leq l(B) \cdot k(D)$. 
	\bigskip
	 
	For $a = 2$, we can use the improved upper bounds
	$$k(D')= \prod_{i \geq 1} k(D_{i,4}')^{a_i} \geq 2^{2 a_1} \cdot \prod_{i \geq 2} \left(2^{1.4 \cdot 2^i -  i + 1}\right)^{a_i} \geq 2^{1.4 w + \sum_{i \geq 2} a_i (-i + 1)  -0.8 a_1}$$ and 
	$$k(D) = \prod_{i \geq 1} k(D_{i,4})^{a_i} \geq \prod_{i \geq 1} \left(2^{1.4 \cdot 2^i + 1}\right)^{a_i} = 2^{1.4 w + \sum_{i \geq 1} a_i}.$$ 
	 With this and the bound from Lemma \ref{lem:bound1.65}, we obtain
	$$k^w(B) \leq 2^{1.4w + 1.65} \leq 2^{1.4w + 3 \sum_{i \geq 2} a_i + 2.2 a_1} \leq k_0(B) \cdot k(D').$$
	
	Here, we used that for $w \geq 4$, there exists an $a_i > 0$ with $i \geq 2$ and that for $w = 2$, we have $a_1 = 1$.
	Since $\pi(w) \geq 2$ for $w \geq 2$, we obtain for (C2)
	$$k(B) =  2^{1.4 w + 1.65} \leq \pi(w) \cdot  2^{1.4w + 1} \leq l(B) \cdot k(D), $$
	so the inequalities hold.
\end{proof}

\subsection{\texorpdfstring{The conjecture for $\varepsilon q \equiv 3$ \text{mod 4}}{The conjecture for eq = 3 mod 4}} 
We examine the case $\ell = 2$ and $\varepsilon q \equiv 3 \mod 4$ by using a recursion to reduce to the previous case. Here, it holds that $a = 1$ and $\tilde{a} \geq 2$ in Equation \ref{eq:formulakB2}. 

\begin{lemma}\label{lem:Sylowb}
It holds that $k(P_1) = 2$, $k(P_2) = 2^{\tilde{a}} +3$, $k(P_4) = k(2^{\tilde{a}} + 3, 2) = 2^{2\tilde{a} - 1} + 9 \cdot 2^{\tilde{a}-1} + 9$ and 
$$k(P_{2^i}) \geq \frac{k(P_4)^{2^{i-2}}}{2^{2^{i-2} - 1}} \geq 2^{(\tilde{a}-1)\cdot 2^{i-1}+1}$$ for $i \geq 2$
as well as $k(P_1') = 1$, $k(P_2') = 2^{\tilde{a}}$ and, for $i \geq 2$, 
$$k(P_{2^i}') \geq \frac{k(P_{2^i})^2}{2^i} \geq 2^{(\tilde{a}-1)\cdot 2^{i-1} -i + 2}.$$
\end{lemma}
\begin{proof}
Note that $P_{2^i}$ lies in $P_{2^{i-1}}^2$ by \cite[Lemma 1.4]{OLS76} with index $2^i$. With this, the proof can be carried out analogously to \cite[Lemma 5.10]{MAL18}. The formula for $k(P_4)$ follows from \cite[Lemma 4.2.9]{JAM84} together with Lemma \ref{lem:k3aw2}.
\end{proof}
\begin{lemma}\label{lem:recb}
Let $c_1, c_2 \in \mathbb{Z}_{>0}$ with $c_1 \geq c_2$ and assume that there exist constants $y, c \geq 0$ such that for every $t \geq 0$, we have 
$$\sum_{ \textbf{w} \in W_t} k(2^{\tilde{a}}-1,w_0) \cdot \prod_{i \geq 1} k(2^{\tilde{a}-1},w_i) \leq 2^{(\tilde{a}-y)t + c}.$$
Then 
$$k(B) \leq 2^{(\tilde{a}-y) \frac{w-a_0}{2} + a_0+ 0.35+c} \cdot \sum_{j = 0}^{(w-a_0)/2} \left(2^{2+y-\tilde{a}}\right)^j.$$
\end{lemma}
\begin{proof}
For $j \in \{0, \ldots, \frac{w-a_0}{2}\}$, let $r(j) = \frac{w-(2j +a_0)}{2}$. Again, we exploit the correspondence between the set of binary decompositions $W_w$ of $w$ and
	$\bigcup_{j = 0}^{(w-a_0)/2} W_{r(j)}$ (see proof of Lemma \ref{lem:bound1.65}). This together with Lemma \ref{lem:multipartitionsbasics} and the assumption yields
\begin{alignat*}{1}
k(B)  &= \sum_{j = 0}^{(w-a_0)/2} k(2,2j+a_0) \cdot \sum_{(w_1, \ldots, w_v) \in W_{r(j)}} k(2^{\tilde{a}}-1,w_1) \cdot \prod_{i \geq 2} k(2^{\tilde{a}-1},w_i) \\
&\leq \sum_{j = 0}^{(w-a_0)/2} 2^{2j+a_0 + 0.35} \cdot 2^{(\tilde{a}-y) \frac{w-(2j+a_0)}{2} + c} = 2^{(\tilde{a}-y) \frac{w-a_0}{2} + a_0+ 0.35+c} \cdot \sum_{j = 0}^{(w-a_0)/2} \left(2^{2+y-\tilde{a}}\right)^j. \tag*{\qedhere}
\end{alignat*}
\end{proof}

\begin{Remark}\label{rem:weven}
Since $k(2,2w_0+1) \leq 2 \cdot k(2,2w_0)$ for all $w_0 \geq 0$ and the sets of binary decompositions of $2j+1$ and $2j$ for $j \geq 0$ are in bijective correspondence by Lemma \ref{lem:bound1.65}, it follows as above that 
\begin{alignat*}{1}
k^{2j+1}(B) &=  \sum_{\textbf{w} \in W_{2j+1}} k(2,w_0) k(2^{\tilde{a}}-1,w_1) \prod_{i \geq 2} k(2^{\tilde{a}-1}, w_i) \\
&\leq \sum_{\textbf{w} \in W_{2j}} 2 \cdot k(2,w_0) k(2^{\tilde{a}}-1,w_1) \prod_{i \geq 2} k(2^{\tilde{a}-1}, w_i)= k^{2j}(B).
\end{alignat*}
\end{Remark}

We can now prove the inequalities of the conjecture. As before, we treat the case $\tilde{a} = 2$ separately.

\begin{lemma}
(C1) and (C2) hold for the principal $2$-block of $\operatorname{GL}_w(\varepsilon q)$ if $\varepsilon q \equiv 3 \mod 4$ and $\tilde{a} \geq 3$. 
\end{lemma}

\begin{proof}
We apply Lemma \ref{lem:recb} using the bound from Lemma \ref{lem:bound1.5}. For $\tilde{a} \geq 4$, we have $2^{3-\tilde{a}} < 1$, hence the geometric series yields  
$$k(B) \leq 2^{(\tilde{a}-1)\frac{w-a_0}{2} + a_0 + 1.85} \cdot \sum_{j = 0}^\infty (2^{3-\tilde{a}})^j  = \frac{2^{(\tilde{a}-1)\frac{w-a_0}{2} + a_0 + 1.85}}{1-2^{3-\tilde{a}}}\leq 2^{(\tilde{a}-1)\frac{w-a_0}{2} + a_0 + 2.85}.$$

The number of conjugacy classes of $D$ is given by 

$$k(D) = \prod_{i \geq 0 } k(P_{2^i}) \geq 2^{a_0} \cdot \prod_{i \geq 1} \left(2^{(\tilde{a}-1)\cdot 2^{i-1} +1 }\right)^{a_i} = 2^{(\tilde{a}-1)\frac{w- a_0}{2} + \sum_{i \geq 0} a_i}.$$

With this, (C2) holds for $w \geq 4$ since then $l(B) \geq \pi(w) \geq 2^{1.85}$. Since the bound for $k(D)$ increases by a factor of 2 when passing from $w = 2j$ to $w = 2j+1$, it remains to consider $w \in \{1,2\}$. For $w = 1$, we have $k(B) = 2 \leq k(D)$. For $w = 2$, it holds that $\pi(2) = 2$ and so $k(B) = 2^{\tilde{a}} + 4 \leq 2^{\tilde{a}+1} \leq l(B) \cdot k(D).$
\bigskip

For the derived subgroup, we have the estimate

$$k(D') = \prod_{i \geq 0} k(P_{2^i}')^{a_i} \geq 2^{\tilde{a} a_1} \cdot \prod_{i \geq 2} \left(2^{(\tilde{a}-1)2^{i-1} -i+ 2}\right)^{a_i} \geq 2^{(\tilde{a}-1)\frac{w-a_0}{2} - \sum_{i \geq 1} (i-2) a_i},$$

so $$k_0(B) \cdot k(D') \geq 2^{(\tilde{a}-1)\frac{w-a_0}{2}  + a_0 + 3 \sum_{i \geq 1} a_i},$$
hence (C1) holds since $\sum_{i \geq 1} a_i \geq 1$ for $w \geq 2$ and $k_0(B) \geq 2 = k(B)$ for $w = 1$. 
\bigskip

We now consider the case $ \tilde{a}= 3$. There, we use the stronger bounds in Lemma \ref{lem:Sylowb} to obtain

\begin{equation}\label{eq:kDb3}
k(D) = \prod_{i \geq 0} k(P_{2^i})^{a_i} \geq 2^{a_0} \cdot 11^{a_1} \cdot \prod_{i \geq 2} \left(2^{1.3 \cdot 2^i +1}\right)^{a_i} \geq 2^{1.3 w- 0.3 a_0 + 0.85 a_1 + \sum_{i \geq 2} a_i}.
\end{equation}
Furthermore, Lemma \ref{lem:recb} yields
\begin{equation}\label{eq:kB3}
k(B) \leq \left(\frac{w-a_0}{2} + 1\right) \cdot 2^{w + 3.35}.
\end{equation}
For $w \geq 11$, this can be bounded from above by $k^w(B) \leq 2^{1.3 w + 2.7}$
since then $(w-a_0)/2 +1 \leq 2^{0.3 w- 0.65}$ by induction. With this, (C2) holds for $w \geq 11$ since $\pi(w) \geq \pi(11) \geq 2^{3.35}$. For (C1), Lemma \ref{lem:Sylowb} yields
%
\begin{equation}\label{eq:kD'b3}
k(D') = \prod_{i \geq 0} k(P_{2^i}')^{a_i} \geq 2^{3 a_1} \cdot 2^{6 a_2} \cdot \prod_{i \geq 3} \left(2^{1.3 \cdot 2^i - (i-2)} \right)^{a_i} = 2^{1.3 (w-a_0) + 0.4 a_1 + 0.8 a_2 - \sum_{i \geq 3} (i-2) a_i}.
	\end{equation}
	For $w \geq 11$, the claim follows with $\sum_{i \geq 2} a_i \geq 1$ and $a_0 \in \{0,1\}$: 
	\begin{equation}\label{eq:dobby}
	k(B) \leq 2^{1.3 w + 2.7} \leq 2^{1.3 w - 0.3 a_0+ 2.4 a_1  + 0.8 a_2 +  3 \sum_{i \geq 2} a_i } \leq k_0(B) \cdot k(D').
	\end{equation}
		
Using the exact values of the $a_i$ in the estimates of Equations \eqref{eq:kDb3}, \eqref{eq:kB3} and \eqref{eq:dobby}, the claim holds for $w \in \{6,10\}$. For the remaining cases, we note that as before, we gain a factor 2 in Equations \eqref{eq:kDb3} and \eqref{eq:kD'b3} when passing from $w = 2x$ to $w = 2x+1$ for some $x \in \mathbb{Z}_{>0}$. So by Remark \ref{rem:weven} it suffices to consider the case $w = 1$ or $w$ even. We obtain the following values
	
\begin{center}
	\begin{tabular}{  l | l | l | l }
		$w$ & $k^w(B)$ & lower bound for $k_0(B) \cdot k(D')$ & lower bound for $l(B) \cdot k(D) $\\ \hline
		$1$ & $2$ & $2$& $2$ \\ 
		$2$ & $12$ & $2^5$ & $2^{4.45}$ \\ 
		$4$ & $94$ & $2^9$ & $5 \cdot 2^{6.2}$\\ 
		$8$& $2908$ & $2^{13.4}$& $22 \cdot 2^{11.4}$\\

	\end{tabular}
\end{center} 
which finishes the proof.
\end{proof}

\begin{lemma}
	(C1) and (C2) hold for the principal $2$-block of $\Glwo$ for $\varepsilon q \equiv 3 \mod 4$ and $\tilde{a} =2$. 
\end{lemma}

\begin{proof}
It holds that $k(3,w) \leq 2^{1.2 w + 0.9}$ (for $w \geq 8$, this follows from Lemma \ref{lem:bound3,4} and the remaining cases can be checked directly). With this, we can prove analogously to Lemma \ref{lem:bound1.65} that 
$$\sum_{ \textbf{w} \in W_w} k(3,w) \cdot \prod_{i \geq 1} k(2, w_i) \leq 2^{1.4 w + 1}.$$ With this, Lemma \ref{lem:recb} yields 
	$$k(B) \leq 2^{1.4 \frac{w-a_0}{2} + a_0 + 1.35} \cdot \sum_{j = 0}^{(w-a_0)/2} 2^{0.6 j} \leq 2^{0.7 w + 0.3 a_0 +1.35} \cdot \frac{2^{0.6 \left(\frac{w-a_0}{2}+1 \right)}-1}{2^{0.6}-1} \leq 2^{w + 2.95}.$$
	
	The number of conjugacy classes of the defect group $D$ is bounded by
	\begin{equation}\label{eq:toffi2}
	k(D) \geq \prod_{i \geq 0} k(P_{2^i})^{a_i} \geq 2^{a_0} \cdot 7^{a_1} \cdot \prod_{i \geq 2} \left(2^{2^i +1}\right)^{a_i} \geq 2^{w + 0.8 a_1 + \sum_{i \geq 2} a_i}
	\end{equation}
	and
	\begin{equation}\label{eq:toffi3}
	k(D') \geq \prod_{i \geq 0} k(P_{2^i}')^{a_i} \geq 2^{2a_1} \cdot 2^{4.45a_2} \cdot \prod_{i \geq 3} \left(2^{2^i -i+2}\right)^{a_i} \geq 2^{w -a_0+0.45a_2 - \sum_{i \geq 3} (i-2) a_i}.
	\end{equation}

	With $l(B) \geq \pi(w) \geq 2^{1.95}$ and $\sum_{i \geq 2} a_i \geq 1$ for $w \geq 4$, (C2) holds. For $w \geq 4$, we obtain for the first inequality 
	\begin{equation}\label{eq:toffi}
	k(B) \leq 2^{w + 0.95} \leq 2^{w + 2 a_1 + 0.45 a_2 + 3 \sum_{i \geq 2} a_i } \leq k_0(B) \cdot k(D').
	\end{equation}
Using the Equations \eqref{eq:toffi2} and \eqref{eq:toffi}, we obtain the following table
	\begin{center}
		\begin{tabular}{  l | l | l | l }
			$w$ & $k^w(B)$ & lower bound for $k_0(B) \cdot k(D')$ & lower bound for $k(D) \cdot l(B) $\\ \hline
			$1$ & $2$ & $2$& $2$ \\
			$2$ & $8$ & $2^4$ & $2^{3.8}$ \\ 
			$3$ & $16$ & $2^5$ & $3 \cdot 2^{3.8}$\\ 
			\end{tabular}
	\end{center}
so the inequalities hold.	
\end{proof}
This completes the proof of (C1) and (C2) for the general linear and unitary groups. 

\section{Special linear and unitary groups}
In the following, assume $\ell = 3$. We prove the conjecture for the special linear and unitary groups $\Sleps$, proceeding similarly to the proof of \cite[Thm.~5.16]{MAL18}. Observe that the proof given therein for the case that $\ell$ does not divide $q-\varepsilon$ is also valid for $\ell = 3$. Therefore, it remains to consider the case $3 |(q-\varepsilon)$. There, $\Gleps$ has a single unipotent block $\tilde{B}$ (cf. \cite[Thm. 7.A]{FON82}) with defect group $\tilde{D}$ which covers the unique unipotent block $B$ of $\Sleps$ (cf. \cite[Thm.]{CAB94}). Let $3^a$ be the exact power of $3$ dividing $q-\varepsilon$ and $3^m := |Z(G)|_3 = \gcd(w, q-\varepsilon)_3 = \min\{\log_3(w_3), a\}.$ 
\bigskip

For $Z \leq Z(G)$, let $B_Z$ be the principal block of $G/Z$ with defect group $D_Z$ and as a special case, let $\bar{B} = B_{Z(G)}$ be the principal block of $\operatorname{PSL}_n(\varepsilon q)$. The bounds given in the proof of \cite[Thm.~5.16]{MAL18} are also valid for $\ell = 3$: It holds that $k(D) \geq k(\tilde{D})/3^a$. For all $Z \leq Z(G)$, we obtain $k(D_Z) \geq k(\bar{D}) \geq k(D)/3^m$ (similarly for the derived subgroups). Moreover, it holds that $k_0(\bar{B}) = k_0(B) \geq k_0(\tilde{B})/3^a$ and $l(\tilde{B}) \geq l(B) = l(B_Z)$ as well as $k(B_Z) \leq k(B).$ In order to prove (C1) and $(C2)$ for the block $B_Z$ for $Z \leq Z(G)$, it is therefore sufficient to prove the following inequalities: 
\begin{equation}
k(B) \leq k_0(B) \cdot k(\bar{D}') \tag{C1'}
\end{equation}
and 
\begin{equation}
k(B) \leq l(B) \cdot k(\bar{D}) \tag{C2'}.
\end{equation}

If $w$ is not divisible by $3$, then $m = 0$. By \cite[Thm.~5.1]{MAL17}, it holds that $k(B) = k(\tilde{B})/3^a$, so it follows from the proven inequalities for the block  $\tilde{B}$ that (C1) and (C2) hold for $B$. We therefore assume that 3 divides $w$.

\begin{lemma}\label{lem:abschaetzungspeciallingroup} 
	Let $a \geq 2$, $w \geq 6$ a positive integer divisible by 3 and $1 \leq j \leq \min\{a, \log_3 w_3\}$. Then it holds that 
	\begin{equation}\label{eq:baum}
	3j + \frac{aw}{3^j} \leq
	\left(a - \frac{5}{6}\right) w. 
	\end{equation}
\end{lemma}

\begin{proof}
	The inequality holds for $w \geq 6$ if $j = 1$ and for $w \geq 9$ if $j = 2$ (note that for $w = 6$, only $j = 1$ is admissible). So assume $j \geq 3$. Using $j \leq a$, the left hand side can be bounded from above by $3a + \frac{aw}{27}$. The resulting inequality
	$$3a  + \frac{aw}{27} \leq \left(a-\frac{5}{6}\right) w$$
	
	is fulfilled for $w \geq \frac{162 a}{52a - 45},$
	i.e., for $w \geq 6$ if $a \geq 2$. 
\end{proof}
The proof of the following lemma is analogous to \cite[Thm.~5.16]{MAL18}.

\begin{lemma}\label{lem:kBsl}
Let $a \geq 2$ and assume that $w \geq 6$ is divisible by 3. It holds that
$$k(B) \leq p_3(w) \cdot 3^{a(w-1) - \frac{5}{6}w + \log_3(19/18) +2} \leq 3^{a(w-1) - \frac{2}{3}w + \log_3(19/18)+2}.$$ 
\end{lemma}

\begin{proof}
	As in the proof of \cite[Thm.~5.16]{MAL18}, we obtain $k(3,3^a,1,x) \leq p_3(x) \cdot 3^{ax}$ for all $x \geq 0$. By \cite[Thm.~5.1]{MAL17}, the number of characters in the block $B$ can thus be bounded by 
	\begin{equation}\label{eq:kBfuerspeciallineargroup}
	k(B) \leq \left(k(3,3^a,1,w) + \sum_{j=1}^{m} p_3\left(\frac{w}{3^j}\right) \cdot 3^{2j + \frac{aw}{3^j}}\right)/3^a,
	\end{equation}
	
	where $3^m = \min\{w_3, 3^a\}$ as before. With Lemma \ref{lem:abschaetzungspeciallingroup} and the bound from Lemma \ref{lem:bound23} we obtain
	\begin{alignat*}{2}
	k(B) &\leq p_3(w) \cdot \left(3^{\left(a-\frac{5}{6}\right)w + 2} + \sum_{j=1}^{m} 3^{2j + \frac{aw}{3^j}}\right)/3^a
	&&\leq p_3(w) \cdot 3^{a(w-1) - \frac{5}{6}w +2} \left(1 + \sum_{j \geq 1} 3^{-2-j}\right) \\
	&\leq p_3(w) \cdot 3^{a(w-1) - \frac{5}{6}w +2} \left(1 + \frac{1}{18}\right) 
	&&= p_3(w) \cdot 3^{a(w-1) - \frac{5}{6}w  +2+ \log_3(19/18)}. 
	\end{alignat*} 
The second bound follows from that with Lemma \ref{lem:plw}.
\end{proof}
	
\begin{lemma}
	The inequalities (C1') and (C2') hold if $a \geq 2$ and $w$ is divisible by 3. 
\end{lemma}

\begin{proof}	
We first assume $w \geq 6$ and consider (C2'). With the estimates from the beginning of this chapter and Equation \eqref{eq:kD3} it holds that
\begin{equation}\label{eq:kDbar}
\kD \geq \frac{k(\tilde{D})}{3^{a+m}} \geq 3^{a(w-1) - \frac{w}{2} - m  + \frac{1}{2} \sum_{i \geq 1} a_i}.
\end{equation}
	
%
%
%
We obtain
\begin{equation}\label{eq:AbschaetzungC2final69}
\frac{\kD \cdot l(B)}{k(B)} \geq \frac{\pi(w) \cdot 3^{a(w-1) - \frac{w}{2} - m + \frac{1}{2} \sum_{i \geq 1} a_i}}{p_3(w) \cdot 3^{a (w-1) - \frac{5}{6}w + \log_3(19/18) +2}} = \frac{\pi(w)}{p_3(w)} \cdot 3^{\frac{w}{3} - m - \log_3(19/18) + \frac{1}{2} \sum_{i \geq 1} a_i -2}. 
\end{equation}
Using $m \leq \log_3(w)$ and $\sum a_i \geq 1$, the above quotient is greater than 1 for $w \geq 12$. This also holds for $w \in \{6,9\}$ when inserting the exact values of $w$, $m$ and $\sum_{i \geq 1} a_i$ and using that $\pi(w)/p_3(w) \geq 3$ in this case. 
\bigskip

Now we consider the inequality (C1'). 
By the estimates from the beginning of this section, Equation \ref{eq:kD'3} and \cite[Prop. 5.15]{MAL18} we have
\begin{equation}\label{eq:kDbar'}
k(\bar{D}') \geq \frac{k(\tilde{D}')}{3^{m + \delta}} \geq 3^{\left(a-\frac{1}{2}\right)w - \sum_{i \geq 1} a_i\left(a+i-\frac{5}{2}\right)-m - \delta},
\end{equation}
where $\delta = 1$ if $w$ is a power of 3 and $\delta = 0$ otherwise. Moreover, we have 
$$k_0(B) \geq \frac{k_0(\tilde{B})}{3^a} \geq 3^{\sum_{i \geq 1} a_i (a+i-1) - a}.$$
	
With Lemma \ref{lem:kBsl} we obtain for $w \notin \{3,9\}$
$$\frac{k_0(B) \cdot k(\bar{D}')}{k(B)} \geq \frac{3^{\left(a-\frac{1}{2}\right)w + \frac{3}{2} \sum_{i \geq 1} a_i -m-a- \delta}}{3^{a(w-1) - \frac{2}{3} w  + \log_3(19/18)+2}} \geq 3^{\frac{w}{6}+ \frac{3}{2} \sum_{i \geq 1} a_i - \log_3(19/18) - m - 2- \delta} > 1.$$
	
The last inequality holds for $w \geq 15$ not a power of 3 since then $m \leq \log_3(w/2) \leq \frac{w}{6}-1/2 - \log_3(19/18)$ and $\sum a_i \geq 1$. For $w \in \{6,12\}$, we have $m = 1$ and $\sum a_i = 2$, so the above term is greater than 1. 
For $w \geq 27$ a power of 3, the inequality holds with $m \leq \log_3(w)$.
\bigskip

Let $w = 9$. For $a \geq 3$, we can use the stronger bound in Lemma \ref{lem:multipartitions} in the proof of Lemma \ref{lem:bound23} to obtain an improved bound $k(3,3^a,1,9) \leq p_3(9) \cdot 3^{9 \cdot \left(a-\frac{5}{6}\right)} \leq 3^{9a-6}.$ With this, the number of characters in $B$ can be bounded by 
$$k(B) \leq \frac{3^{9a -6} +2 \cdot 3^{2+3a} + 3^{4+a} }{3^a} \leq 2 \cdot 3^{8a-6}.$$
	
By \cite[Example 5.14]{MAL17} we have $k_0(B) = 18$ and Equation \eqref{eq:kDbar'} yields $k(\bar{D}') \geq 3^{8a-7}$, so the inequality holds.  
For $a = 2$ and $w  =9$, it holds that $k(B) \leq 45687$. By the above calculation, we have $k(\bar{D}') \geq 3^9$ and $k_0(B) = 18$, so (C1') holds. 
\bigskip
	
It remains to consider both inequalities for $w = 3$. In this case, we have $m = 1$ and  Equations \eqref{eq:kDbar'} and \eqref{eq:kDbar} yield $k(\bar{D}'),\, \kD \geq 3^{2a-2}$. Furthermore, by \cite[Example 5.14]{MAL18}, it holds that $k_0(B) = 6$ and $l(B) \geq 5$. By Lemma \ref{lem:k3aw2}, we have 
$$k(3^a,3)= \frac{1}{6} \cdot 3^{3a} + \frac{3}{2} \cdot 3^{2a} + \frac{4}{3} \cdot 3^a \leq 0.35 \cdot 3^{3a}, $$

hence
\begin{equation*}
k(B) = \frac{k(3^a,3) + 3^{2+a} - 3^{a-1}}{3^a} 
\leq 0.35 \cdot 3^{2a} +9 - \frac{1}{3} \\
\leq 5 \cdot 3^{2a-2},
\end{equation*}

where we used the assumption $a \geq 2$ in the third step.
This yields
$$k(B) \leq 0.5 \cdot 3^{2a} \leq 5 \cdot 3^{2a-2} \leq \min\{k_0(B) \cdot k(\bar{D}'),\, l(B) \cdot k(D)\},$$
so both inequalities hold also in this case. 
\end{proof}

\begin{lemma}\label{lem:SlC2a1w3}
	The inequalities (C1') and (C2') hold if $a = 1$ and $w \geq 6$ is divisible by 3.  
\end{lemma}

\begin{proof}
It holds that $m \leq a = 1$ in this case. First consider (C2'). With the improved bound from Equation \eqref{eq:kD31}, we have
\begin{equation}\label{eq:kDa1wleq10}
\kD \geq \frac{k(\tilde{D})}{9} \geq 3^{\frac{2}{3} w  + \frac{1}{2}\sum_{i \geq 1} a_i -2}.
\end{equation}

By Lemmas \ref{lem:bound1.65} and \ref{lem:plw}, it holds

\begin{equation}\label{eq:kbsla1}
k(B) \leq \frac{k(3,3,1,w) + p_3(w/3) \cdot 3^{2 + \frac{w}{3}}}{3}
\leq 3^{\frac{w}{2} + \frac{5}{2}} + 3^{1 + \frac{w}{2}} \leq 3^{\frac{w}{2}+ 2.67}.
\end{equation}

For $w \geq 15$, the number of irreducible Brauer characters in $B$ can be bounded by $l(B) \geq \pi(w) \geq 3^{4.67}$, so (C2') holds. Using the exact values in Equation \eqref{eq:kDa1wleq10}, the claim also holds for $w \in \{9,12\}.$ 
For $w = 6$, we have $\lB \geq k(1,6) = 11$, $\kD \geq 3^{3}$ as well as $k(B) \leq \left(k(3,3,1,6)  +3^4\right) /3 = 117$.
\bigskip

Now consider (C1'). Let $\delta = 1$ if $w$ is a power of 3 and zero otherwise as before. Equation \eqref{eq:kD'31} yields
\begin{equation}\label{eq:kD'}
k(\bar{D}') \geq \frac{k(\tilde{D}')}{3^{1+\delta}} \geq 3^{\frac{2}{3} w+ \sum_{i \geq 1} a_i \left(\frac{1}{2}-i\right)-1 - \delta}.
\end{equation}

With the improved bound from Lemma \ref{eq:betterboundk0B} we obtain 
$$\frac{k(\bar{D}') \cdot k_0(B)}{k(B)} \geq \frac{3^{\frac{2}{3} w+ \sum_{i \geq 1}  a_i \left(\frac{1}{2}-i\right) -1 - \delta + \sum_{i \geq 1} a_i \cdot i + \sum_{i: a_i \neq 0} 1 -1}}{3^{2.67 + \frac{w}{2}}} \\
= 3^{\frac{w}{6} - 4.67 + \frac{1}{2}\sum_{i \geq 1} a_i + \sum_{i: a_i \neq 0} 1 - \delta}.$$
	
If $w$ is not a power of 3, we either have $\sum_{i \geq 1} a_i \geq 2$ or $\sum_{i: a_i \neq 0} 1 \geq 2$. With this, the inequality holds for $w \geq 18$. Using the exact values of the $a_i$, the claim follows for $w \in \{12,15\}$. For $w =6$, we have $k_0(B) \geq k_0(\tilde{B})/3 =  k(3^3,2)/3= 135$ by Equation \eqref{eq:betterboundk0B} and $k(B) \leq 117.$ For $w \geq 27$ a power of $3$, the above term is greater than 1. For $w = 9$, we have $k(B) \leq 745$ by Equation \ref{eq:kBfuerspeciallineargroup}. It holds that $|D_{1,3}^3 : D_{2,3}'| = 9$ (cf. \cite[Lemma 5.10]{MAL18}) and we can check directly that $k(D_{1,3}) = 17$, so we have 
$$k(\bar{D}') \geq \frac{k(\tilde{D}')}{9} = \frac{k(D_{1,3})^3}{81} \geq 3^{3.5}.$$
With $k_0(B) = 18$ (cf. \cite[Thm.~5.12]{MAL17} the inequality holds in this case. 
\end{proof}
For the remaining case $a = 1$ and $w = 3$, we consider the original inequalities (C1) and (C2).

\begin{lemma}
	The inequalities (C1) and (C2) hold for the principal block of $H = G/Z$ if $a =1$ and $w = 3$. 
\end{lemma}

\begin{proof}
	We use the notation from the beginning of this chapter. We can check directly that $|\tilde{D}| = 81$, hence $|D| = 27$. Since $|Z(G)|_3 = 3$, a defect group $D_Z$ for $Z \leq Z(G)$ is either isomorphic to $D$ or $|D_Z| = 9$. In the latter case, $D_Z$ is abelian and the claim holds by \cite[Thm.~2.1]{MAL18}. For the first case, since $k_0(B) = k_0(B_Z)$ and $l(B) = l(B_Z)$, it suffices to prove the inequality for $\Slweps$. We have $|\tilde{D}'| = 9$ and hence $|D'| = 3$, thus $k(D') = 3$. By Example 5.14 in \cite{MAL18}, we have $k_0(B) = 6$ and $k(B) = 16$, so (C1) holds. For (C2), we use Example 5.14 in \cite{MAL18} to obtain $l(B) = 5$ and $k(D) \geq k(\bar{D}) = 9$, since $\bar{D}$  is abelian. With this, the inequality holds. 
\end{proof}

This completes the proof of our main theorem. 
\bigskip

\textbf{Acknowledgement:} I would like to thank Prof.\  G.\ Malle for supervising my master thesis as well as for his numerous suggestions and his advice concerning this project. 

\nocite{GAP4}
\nocite{BRE19}

\small
\bibliographystyle{plainurl}
\bibliography{Biblio.bib}
\bigskip
\normalsize

Institut f\"ur Mathematik, Friedrich-Schiller-Universit\"at, 
Ernst-Abbe-Platz 2,
07743 Jena, Germany \\
E-mail: sofia.bettina.brenner@uni-jena.de


\end{document}